\documentclass[12pt]{article}
\usepackage{amsmath,amssymb,amsthm, amsfonts}
\usepackage{etaremune, enumerate, float, verbatim}
\usepackage{hyperref}
\usepackage{graphicx}
\usepackage{url}
\usepackage[mathlines]{lineno}
\usepackage{dsfont} % needed for double-struck 1
\usepackage{tikz, graphicx, subcaption, caption}% added CR 12/6/17
\usepackage[algo2e,ruled,vlined]{algorithm2e}
\usepackage{color}
\definecolor{red}{rgb}{1,0,0}

\definecolor{blue}{rgb}{0,0,.7}

\definecolor{green}{rgb}{0,.6,0}

\definecolor{purp}{rgb}{.5,0,.5}

\numberwithin{figure}{section}   % added LH 11/15/17

\newtheorem{thm}{Theorem}[section]
\newtheorem{cor}[thm]{Corollary}
\newtheorem{lem}[thm]{Lemma}
\newtheorem{prop}[thm]{Proposition}

\theoremstyle{definition}
\newtheorem{rem}[thm]{Remark}

\theoremstyle{definition}

\theoremstyle{definition}

\newcommand{\Z}{\operatorname{Z}}
\newcommand{\Zp}{\operatorname{Z}_+}
\newcommand{\Zs}{\operatorname{Z}_-}
\newcommand{\thz}{\operatorname{th}}
\newcommand{\thp}{\operatorname{th}_+}
\newcommand{\thc}{\operatorname{th}_c}
\newcommand{\ths}{\operatorname{th}_-}
 
\newcommand{\dist}{\operatorname{dist}}  
\newcommand{\pt}{\operatorname{pt}}
\newcommand{\ptp}{\operatorname{pt}_+}
\newcommand{\pts}{\operatorname{pt}_-}

\newcommand{\wh}{\widehat}
\newcommand{\wt}{\widetilde}
\newcommand{\bit}{\begin{itemize}}
\newcommand{\eit}{\end{itemize}}
\newcommand{\ben}{\begin{enumerate}}
\newcommand{\een}{\end{enumerate}}
\newcommand{\beq}{\begin{equation}}
\newcommand{\eeq}{\end{equation}}
\newcommand{\bea}{\begin{eqnarray*}} % * means no number
\newcommand{\eea}{\end{eqnarray*}}
\newcommand{\bpf}{\begin{proof}}
\newcommand{\epf}{\end{proof}\ms}
\newcommand{\bmt}{\begin{bmatrix}}
\newcommand{\emt}{\end{bmatrix}}
\newcommand{\ms}{\medskip}

\newcommand{\lc}{\left\lceil}
\newcommand{\rc}{\right\rceil}
\newcommand{\lf}{\left\lfloor}
\newcommand{\rf}{\right\rfloor}
\newcommand{\du}{\,\dot{\cup}\,}
\newcommand{\noi}{\noindent}
\newcommand{\lp}{\!\left(}
\newcommand{\rp}{\right)}

\title{Skew throttling}
\author{Emelie Curl\thanks{Department of Mathematics, Christopher Newport University, 1 Avenue of the Arts, Newport News, VA 23606, emelie.curl@cnu.edu.}\and Jesse Geneson\thanks{Department of Mathematics, Iowa State University, Ames, IA 50011, USA, (geneson,  hogben)@iastate.edu.}\and Leslie Hogben\footnotemark[2]\ \thanks{American Institute of Mathematics, 600 E. Brokaw Road, San Jose, CA 95112, USA, hogben@aimath.org}
}
\begin{document}
%\linenumbers
\maketitle

\begin{abstract} % edited 5/16/19 LH
 Zero forcing is a process that colors the vertices of a graph blue by starting with some vertices blue and applying a color change rule. Throttling minimizes the sum of the number of initial blue vertices and the time to color the graph. In this paper, we study throttling for skew zero forcing. We characterize the graphs of order $n$ with skew throttling numbers $1, 2, n-1$, and $n$. We  find the exact  skew throttling numbers of paths, cycles, and balanced spiders with short legs. In addition, we find a sharp lower bound on skew throttling numbers in terms of the diameter.  \end{abstract}
 
\noi {\bf Keywords} skew zero forcing, skew propagation time, skew throttling\ms

\noi{\bf AMS subject classification} 05C57, 05C15, 05C50

%%%%%%%%%%%%%%%%%%%%%%%%%%%%%%%%%%%%%%%%%%%

\section{Introduction}\label{s:intro}
Zero forcing is a process on graphs in which vertices have two possible colors, blue and white. In each {\em round} (also called a time step), the current blue vertices with only one white neighbor will color that neighbor. An initial set $S$ of blue vertices that eventually colors the whole graph blue is called a \emph{zero forcing set}. For any graph $G$, the minimum possible size of a zero forcing set is called the zero forcing number $\Z(G)$. The zero forcing number of any graph $G$ gives an upper bound for the maximum nullity of the family of symmetric matrices with off-diagonal nonzero  pattern described by the edges of $G$ \cite{AIM08}. Zero forcing has also been applied to quantum systems control \cite{BG,S} and graph searching \cite{Y}. For any initial set $S$ of blue vertices, the number of rounds for the whole graph $G$ to be colored blue is denoted $\pt(G;S)$, the propagation time of $S$. The propagation time of $G$ is the minimum value of $\pt(G;S)$ over all minimum zero forcing sets $S$ \cite{proptime}.

Several other variants of zero forcing have been studied, including {positive semidefinite (PSD) zero forcing} and {skew zero forcing}. In {\em PSD zero forcing}, each blue vertex colors any vertex that is its only white neighbor in a connected component obtained by removing all of the blue vertices. The {\em PSD zero forcing number} $\Zp(G)$ \cite{smallparam} and {\em PSD propagation times} $\ptp(G;S)$ and $\ptp(G)$ \cite{PSDpropTime} are defined analogously to $\Z(G)$, $\pt(G;S)$, and $\pt(G)$. Like $\Z(G)$, $\Zp(G)$ gives an upper bound for the maximum nullity of the family of positive semidefinite matrices corresponding to $G$ \cite{smallparam}. The PSD zero forcing number has also been applied to study the cop versus robber game on trees \cite{CR1}. In {\em skew zero forcing}, every vertex with only one white neighbor colors that neighbor blue in each round. This differs from standard zero forcing in that a white vertex is allowed to color its neighbor. The {\em skew zero forcing number} $\Zs(G)$ \cite{IMA} and {\em propagation times} $\pts(G;S)$ and $\pts(G)$ \cite{King15} are defined in analogy with $\Z(G)$, $\pt(G;S)$, and $\pt(G)$.  As in the case of $\Z(G)$ and $\Zp(G)$, $\Zs(G)$ gives an upper bound for the maximum nullity of the family of skew-symmetric matrices corresponding to $G$ \cite{IMA} and for the maximum nullity of zero-diagonal symmetric matrices corresponding to $G$ (the maximum nullity of weighted adjacency matrices) \cite{mr0}.

For each variant of zero forcing, the propagation time of a graph $G$ is defined only using minimum zero forcing sets, but it is natural to investigate the propagation time for larger sets and to minimize the sum of the number of initially blue vertices and the propagation time of that set.  This is called {\em throttling}.  Throttling minimizes the sum of the resources  and the time needed to accomplish the task.  For a graph $G$  and set  $S\subseteq V(G)$, define $\thz(G;S) = |S| + \pt(G;S)$ and $\thz(G) = \min_{S \subseteq V(G)} \thz(G;S)$. The {\em zero forcing throttling number} $\thz(G)$ was introduced in \cite{BY13throttling}, where a tight lower bound was presented. Throttling numbers $\thp(G)$ and $\thc(G)$ have also been defined analogously for PSD zero forcing in \cite{carlson2019throttling} and the cop versus robber game in \cite{CR1}, where it was proved that $\thp(T) = \thc(T)$ for trees $T$. 

In this paper, we introduce the study of throttling for skew zero forcing. For  a graph $G$ and set $S\subseteq V(G)$,   
define $\ths(G;S) = |S| + \pts(G;S)$ and the {\em skew throttling number} $\ths(G) = \min_{S \subseteq V(G)} \ths(G;S)$. % $\ths(G;S)$ and $\ths(G)$ analogously to $\thz(G;S)$ and $\thz(G)$ respectively. 
For $k\ge \Zs(G)$, it is also convenient to define $\ths(G,k)=\min_{|S|=k} \ths(G;S)$; with this notation, $\ths(G) = \min_{k} \ths(G,k)$. In Section \ref{s:ext}, we characterize the graphs of order $n$ with skew throttling numbers of $1$, $2$, $n-1$ and $n$. In Section \ref{s:path-cycle}, we determine skew throttling numbers for several families of graphs including paths, cycles, and some spiders. We also prove a  lower bound  $\ths(G) = \Omega(\sqrt{d})$ for graphs $G$ of diameter $d$ and minimum degree at least two, and exhibit a family of graphs that achieve this bound.

We define some graph terminology that is used in our results. A \emph{cograph} is a graph that can be generated from $K_1$ using only complementation and disjoint union. Equivalently, a cograph is a graph that does not contain $P_4$ (a path on four vertices)  as an induced subgraph. The \emph{corona} $G_1 \circ G_2$ of $G_1$ with $G_2$ is obtained by making one copy of $G_1$, $|V(G_1)|$ copies of $G_2$, and connecting every vertex in the $i^{th}$ copy of $G_2$ to the $i^{th}$ vertex of $G_1$. A \emph{spider} is a tree with a single vertex of degree at least $3$, which is called the {\em center}. The graph obtained by removing this vertex is a disjoint union of paths. Each of these paths is a leg of the spider, and the length of the leg is the number of vertices in the path. The spider is called \emph{balanced} if all legs have the same length. 

%%%%%%%%%%%%%%%%%%%%%%%%%%%%%%%%%%%%%%%%%%%

\section{Extreme skew throttling numbers}\label{s:ext}

In this section we characterize graphs with very low or very high skew throttling numbers. Butler and Young \cite{BY13throttling} show that $\lc2\sqrt n -1\rc\le \thz(G)$ for all graphs $G$ of order $n$.  However,
 there are in general no useful bounds  on the skew throttling number in terms of the order of the graph, since we exhibit graphs $G$ of order $n$ with $\ths(G)=1$ and   $\ths(G)=n$ %, although every graph with an edge satisfies $\ths(G)\le n-1$. 
and  we show that there are connected graphs $G$ of order $n\ge 3$ with $\ths(G)=2$ and  $\ths(G)=n-1$. % for all  $n\ge 3$.   

\subsection{Low skew throttling} %-------------- -------------------------------------------------------

In this section, we determine graphs having skew throttling number at most two.  %Recall that $\delta(G)-1\le \Zs(G)$ for every graph $G$.

\begin{prop}\label{p:th-1}
For a graph $G$, $\ths(G)=1$ if and only if $G=K_1$ or $G=rK_2$ for $r\ge 1$.  
\end{prop}
\bpf If $\ths(G)=1$, then  $\ths(G,1)=1$ or $\ths(G,0)=1$, which imply   $G=K_1$ or $G=rK_2$, respectively. The converse is clear. \epf

\begin{lem}\label{l:tp-02}
A graph $G$ has $\ths(G)=\ths(G,0)=2$ %can be skew-throttled in two time steps with no initial blue vertices 
if and only if $G=\lp \wh G\circ K_1\rp\du rK_2$ where  $\wh G$ is a  graph of order at least two in which each component of 
$\wh G$ has an edge and $r$ is a nonnegative integer.
\end{lem}
\begin{proof}
Suppose $G=\lp \wh G\circ K_1\rp\du rK_2$, $|V(\wh G)|\ge 2$,   and each component of $\wh G$ has an edge.  
Each leaf  (vertex of degree one) forces its neighbor in the first round.  %Since each component of $\wh G$ has an edge, there are at least two white vertices remaining. 
Then each vertex in $\wh G$  forces its one leaf neighbor, so $\ths\lp\lp \wh G\circ K_1\rp\du rK_2,0\rp=2$. Since the order of a component of $\wh G\circ K_1$ %that contains an edge of $\wh G$ 
is at least four, $\ths\lp\lp \wh G\circ K_1\rp\du rK_2\rp\ne 1$ by Proposition \ref{p:th-1}.  Thus $\ths\lp\lp \wh G\circ K_1\rp\du rK_2\rp=\ths\lp\lp \wh G\circ K_1\rp\du rK_2,0\rp=2$.

Now assume $\ths(G)=\ths(G,0)=2$.  Let $L$ be the set of leaves of $G$.  With $S=\emptyset$, the vertices in $L$ are the only vertices that can force during the first round. Any $K_2$ component of $G$ is now all blue, but $G$ is not.  Define $G'$ to be the subgraph of components of $G$ that are not entirely blue, $L'$ to be the set of leaves of $G'$, and $U=\{ u:u\in N(\ell) \mbox{ for some } \ell\in L'\}$.  No vertex in $L'$ is blue after the first round because $\deg u\ge 2$ for every $u\in U$. In the next round all vertices in $L$ must be colored blue.  This means that the neighbor $u$ of $\ell$ must force $\ell$ in the second round, so every other neighbor of $u$ must be blue after the first round.  Thus $G'=G[U]\circ K_1$ and $G=G'\du rK_2$.  
\epf

For nonnegative integers $s$ and $t$, % such that $s+t\ge 0$, 
define $H(s,t)$ to be the graph with  $V(H(s,t))=\{b\}\du\{x_i,y_i: i=1,\dots,s\}\du \{z_i,w_i:i=1,\dots,t\}$ and $E(H(s,t))=\{bx_i, x_iy_i:i=1,\dots,s\}\cup\{bz_i, bw_i, z_iw_i:i=1,\dots,t\}$.

\begin{lem}\label{l:tp-11}
 A graph $G$ has $\ths(G)=\ths(G,1)=2$ if and only if $G=H(s,t)\du rK_2$ for some $r,s,t\ge 0$ with $r+s+t \geq 1$.  In this case, the order of $G$ is odd.
\end{lem}

\begin{proof}
It is straightforward to verify that $\ths(H(s,t)\du rK_2,\{b\})=2$ for $r+s+t \geq 1$.  By Proposition \ref{p:th-1}, $\ths(H(s,t)\du rK_2)\ge 2$ when $r+s+t \geq 1$.

Assume that $\ths(G) = 2$ and that $G$ can be skew-throttled in one round with one initial blue vertex $b$. Let $\wt G$ be the connected component containing $b$.  Since $K_2$ is the only connected graph that can force itself in one round with no blue vertices, $G=\wt G\du rK_2$ for some $r\ge 0$.  If $|V(\wt G)|=1$, then $\ths(G) = 2$ implies $r\ge 1$ and $G=H(0,0)\du rK_2$.  It is not possible to have $|V(\wt G)|=2$, because this would imply $G=(r+1)K_2$ and $\ths(G)=1$.  If $|V(\wt G)|=3$, then $\wt G=P_3=H(1,0)$ or $\wt G=K_3=H(0,1)$. %, so assume $|V(\wt G)| \geq 4$. 

So assume  $|V(\wt G)| \geq 4$. %No neighbor of $b$  can be a leaf,  since $|V(\wt G)| \geq 3$ implies $b$ cannot force a leaf in the first round. 
Let $v$ be a vertex at maximum distance from $b$ in $\wt G$. If $\dist(b,v)\ge 3$, then for any neighbor $u$ of $v$, $\deg u\ge 2$ and $b$ is not a neighbor of $u$. Thus, $\dist(b,v)\ge 3$ implies $v$ cannot be colored blue in the first round. So no vertex in $\wt G$ is at distance more than  two from $b$. Since $|V(\wt G)| \geq 4$, this implies $\deg b\ge 2$ and $b$ cannot perform a force in the first round. Indeed, if we had $\deg b = 1$, then the only neighbor $u$ of $b$ is a universal vertex in $\wt G$. Since $u$ has at least two neighbors besides $b$, no neighbor of $u$ other than $b$ would ever get colored, so we conclude that $\deg b \neq 1$.
Since $b$ cannot force, forcing in $\wt G$ is the same as forcing in $\wt G-b$.  Thus $\ths(\wt G-b,0)=1$, so $\wt G-b=q K_2$.  For a $K_2$ that has one edge between it and $b$, we designate its vertices as $x_i,y_i$ whereas  a $K_2$ that has two edges between it and $b$ has its vertices designated as $z_i,w_i$; with the vertices labeled this way we have identified $\wt G$ as some $H(s,t)$ with $s+t\ge 2$.  \end{proof}
%{p:ths=n-1}

Since $\ths(G)=2$ implies one of $\ths(G,0)=2$, $\ths(G,1)=2$, or $\ths(G,2)=2$, the next result follows from Lemmas \ref{l:tp-02} and  \ref{l:tp-11} and the observation that $\ths(G,k)=k$ implies $|V(G)|=k$.

\begin{thm}\label{t:ths2}  A graph $G$ has $\ths(G)=2$ if and only if $G$ is one of $2K_1$, $H(s,t)\du rK_2$ with $r+s+t \geq 1$, or $\lp \wh G\circ K_1\rp\du rK_2$  where each component of $ \wh G$ has an edge and $ \wh G$ has order at least two.
\end{thm}

By considering $G=C_m\circ K_1$ we see that $\ths(G)=2$ can be achieved by a graph of arbitrarily large order with maximum degree three, unlike the case of PSD throttling \cite{carlson2019throttling}.   Theorem \ref{t:ths2} also implies that graphs in a well-known family  have  skew throttling number equal to two: 
For $n\ge 1$, the {\em friendship graph} $F_n$ is the planar graph with $2n+1$ vertices and $3n$ edges  constructed by joining a universal vertex to $n$ disjoint copies of  $K_2$; $F_2$ is shown in \ref{f:F2}.  That is, $F_n = H(0, n)$, so $\ths(F_n)=2$ by Lemma \ref{l:tp-11}.  The proof of Lemma \ref{l:tp-11} also showed $\Zs(F_n) = 1$ and $\pts(F_n) = 1$.\vspace{-5pt}
%In the next result we also establish the skew zero forcing number and propagation time of friendship graphs {\red [verify that this is not already in the literature]}

\begin{figure}[h]
\centering
\includegraphics[width=0.3\textwidth]{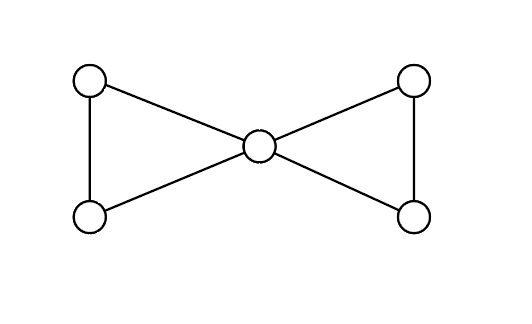}\vspace{-15pt}
\caption{The friendship graph $F_2$\label{f:F2}}%\vspace{-5pt}

\end{figure}

\subsection{High skew throttling} %-------------- -------------------------------------------------------

We now turn to graphs with high skew throttling number. For graphs $G$ with all vertices isolated, it is clear that $\ths(G)=n$. In the next result, we prove an upper bound on skew throttling number for graphs that have an edge. 

\begin{prop}\label{p:ths=n-1} Let $G$ be a graph of order $n$.  If $G$ has an edge, then $\ths(G)\le n-1$.  Thus $\ths(G)=n$ if and only if $G=rK_1$.  %The following graphs  have $\ths(G)=n-1$: $K_n$, $K_{1,n-1}$.
\end{prop}
\bpf  If $G$  has an edge $uv$, then  $\pts(G;V(G)\setminus \{u,v\})=1$, which implies $\ths(G)\le\ths(G;V(G)\setminus \{u,v\})= n-1$.  Thus $\ths(G)=n$ if and only if $G=rK_1$.  %The statements about spcidic graphs are straightforward to verify
\epf

\begin{rem}\label{r:small-pt} Let $G$ be a graph of order $n$ that has an edge, so $\ths(G)\le n-1$.  This implies $\pts(G;S)\ge 1$ for any set $S$ such that $\ths(G;S)=\ths(G)$, which then implies $|S|\le n-2$ for any such $S$. Furthermore, it is straightforward to see that $\pts(G)\le 2$ implies $\ths(G)=\Zs(G)+\pts(G)$: This is immediate for $\pts(G)=1$.  In the case $\pts(G)= 2$, it is not possible to improve  throttling  by adding one to the zero forcing set.
\end{rem}

Let  $G$ be a cograph.  The {\em $\cup-\vee$ decomposition tree} $T_G$ of $G$ is a rooted binary tree such that the vertices of $G$ are the leaves of $T_G$, each non-leaf vertex is labeled either $\cup$ or $\vee$,  %Then $G$ can be constructed from its vertices by applying the indicated operations in order, 
where $\cup$ represents disjoint union and $\vee$ represents join.  
For a non-leaf vertex $x$ of $T_G$, the branches  at $x$ are the two connected components of $T_G$ induced by  the descendants of $x$. %The graph $G$ can be constructed from it vertices by applying the operations that label the non-leaf vertices of $T_G$ to the induced subgraphs of $G$ corresponding to their branches, working from the leaves toward the root. 
If $y$ is a vertex of $G$ (and a leaf of $T_G$), define $G_y=G[\{y\}]$.   For $x$ a non-leaf vertex of $T_G$, define $G_x$ to be the subgraph of $G$ induced by the leaves of $T_G$ that are descendants of $x$. Observe that $G_x$ can be obtained by applying the operation in the label of $x$ to $G_y$ and $G_z$, where $y$ and $z$ denote the immediate descendants of $x$, and $G=G_r$ where $r$ is the root of $T_G$.

If $G$ is a cograph with no induced $2K_2$, then every $\cup$ vertex in the $\cup-\vee$ decomposition of $G$ has a branch with no $\vee$, since otherwise each of the disjoint subgraphs of $G$ induced by the descendants of the $\cup$ vertex would have a  $K_2$.

\begin{thm}\label{t:high-th}
For a graph $G$ of order $n$, $\ths(G) = n-1$ if and only if $G$ is a cograph with no induced $2K_2$ and at least one edge.
\end{thm}

\begin{proof}
Let $G$ be a graph of order $n$.  We first establish that $\ths(G)\ne n-1$ if $G$ has no edges, or if $G$  has an induced $P_4$ or $2K_2$. If $G$ has no edges, then $\ths(G) = n$. If $G$ has an induced $P_4$, then let $S$  consist of all vertices except those in an induced $P_4$, so $\pts(G;S)= 2$ and $\ths(G)\le n-2$. If $G$ has an induced $2K_2$, then let $S'$  consist of all vertices except those in an induced $2K_2$, so $\pts(G;S')= 1$ and $\ths(G)\le n-3$. %That completes the first direction.

Next, we prove by induction on the order of the graph that every cograph $G$ of order $n$ with no induced $2K_2$  has throttling number $n$ if $G$ has no edges and $n-1$ if $G$ has at least one edge. % by induction on the $\cup-\vee$ decomposition tree $T$, starting from the leaves. 
Clearly the statement is true for graphs of order $1$, since there are no edges and the throttling number is $1$.
The induction hypothesis is that every cograph of order $k<n$ with no induced $2K_2$ must have throttling number $k-1$ if it has an edge and $k$ if it has no edge. Let $G$ be a cograph of order $n>1$ with no induced $2K_2$.

Let  $x$ be the root of $T_G$ (so $G_x=G$), and  denote the descendants of $x$ by $y$ and $z$.  The induction hypothesis applies to $G_y$ and $G_z$ since each has order less than $n$.  

Suppose first that $x$ is labeled with $\cup$.  Then at least one of the branches of $x$, say the one that contains $y$, has no $\vee$, so $G_y$ consists of isolated vertices. If neither branch has a $\vee$, then $G_x$ consists of isolated vertices and $\ths(G_x) = |G_x|$. Suppose  the branch at $z$ has a $\vee$, so $G_z$ has an edge. Any zero forcing set of $G_x$ must consist of a zero forcing set for $G_z$ along with every vertex in $G_y$. Thus \[\ths(G_x) = \ths(G_z)+|G_y| =|G_z|-1+|G_y| = |G_x|-1.\]

Now suppose $x$ is labeled with $\vee$.  Let $S$ be a set of vertices such that $\ths(G_x) = \ths(G_x, S)$. %Start with $S$ as the initial set of blue vertices. 
The number of white vertices (i.e., vertices not in $S$) is at least $2$ by Remark \ref{r:small-pt}. % since $G_x$ can be throttled with initial blue set containing all but two adjacent vertices with propagation time $1$.
No vertex in $G_y$ can force any other vertex in $G_y$ until every vertex in $G_z$ is blue, and vice versa. Moreover, no vertex in $G_y$ can force any vertex in $G_z$ until all but one vertex in $G_z$ is blue, and vice versa. For the zero forcing process to start, $G_y$ or $G_z$ must initially have at most one white vertex. If each of $G_y$ and $G_z$ has exactly one white vertex at the start, then $\pts(G_x,S)=1$, and $\ths(G_x) = |G_x|-1$. 

If initially $G_y$ has one white vertex and $G_z$ has more than one white vertex, then no vertex in $G_z$ can be colored blue in the first round: Every vertex in $G_y$ has at least two white neighbors in $G_z$ and a vertex in $G_z$ that has a white neighbor in $G_z$ has at least two white neighbors including one in $G_y$. So in the first round there is exactly one force $u\to w$ where $u\in G_z$ and $w\in G_y$, and all vertices of $G_y$ are blue after the first round. Thus the initial set $S$ has the same throttling number as $S'= S\cup \{w\}$, since throttling with $S'$ adds one to the size of the zero forcing set but subtracts one from the propagation time.  Thus we  replace $S$ by $S'$ for the rest of the proof; observe that $S' \cap V(G_z)=S\cap V(G_z)$.  

If $G_z$ had only isolated vertices, then all but one vertex in $G_z$ would have to be in $S'$, or else no vertex in $G_y$ could color any vertex in $G_z$. However, we assumed that  $S'$ omits  more than one white vertex of $G_z$, so $G_z$ must have an edge.  Since $G_z$ has an edge, $\ths(G_z) =  |G_z|-1$ by induction hypothesis. % It suffices to show that %
Suppose that as $S'$ colors all the vertices blue, a vertex $v$ in $G_y$  forces a vertex $w'$ of $G_z$. Necessarily $v\to w'$ is the last force, and this force is the only force in the last round.  Thus  $S'$ has the same throttling number as $S''=S'\cup \{w'\}$, since throttling with $S''$ adds one to the size of the zero forcing set but subtracts one from the propagation time. In this case,  we  replace $S'$ by $S''$ for the rest of the proof (if this case does not apply, let $S''=S'$).  Define $Z=S''\setminus V(G_y)$.   Using $S''$, no vertex of $G_y$ performs a force, so $Z$ is a skew zero forcing set for $G_z$.  Thus $\pts(G;S'')=\pts(G_z,Z)$ and 
\[\ths(G)=\ths(G;S'')=|G_y|+|Z|+\pts(G_z,Z)=n-1.\qedhere\]
%If $\deg_{G_z} w\ge 1$, then we can instead have one of its neighbors in $G_z$ perfoirm the force, so $Z$ is a zero forcing set for $G_z$ and the result is established. If $w$ is an isolated vertex in $G_z$, then reoplace $S$ by $S'%Note that no vertex in $G_y$ can perform a force until all but one vertex in $G_z$ is blue; denote the last .  At that stage, either 
\end{proof}

It follows from Theorem \ref{t:high-th} that the complete multipartite graph $K_{n_1,n_2,\ldots,n_s}$ with $s\ge 2$ and  $n:=n_1+n_2+\dots+n_s$ is example of a graph  with $\ths(K_{n_1,n_2,\ldots,n_s}) = n-1$; this also follows from results of Kingsley, who showed in \cite{King15} that $\Zs(K_{n_1,n_2,\ldots,n_s})=n-2$ and $\pts(K_{n_1,n_2,\ldots,n_s})=1$.%\ms

%%%%%%%%%%%%%%%%%%%%%%%%%%%%%%%%%%%%%%%%%%%

\section{Skew throttling numbers of families of graphs}\label{s:path-cycle}

In this section we determine the skew throttling numbers of hypercube graphs, paths, cycles, and balanced spiders with a constant number of legs. We also bound the maximum and minimum skew throttling numbers of trees of order $n$ up to a constant factor, as well as trees of a given maximum degree. 

Just as connected graphs of order $n$ have skew throttling numbers ranging between $2$ and $n-1$, the same is true of trees of order $n$. The maximum is at most $n-1$ by Proposition \ref{p:ths=n-1}, and is $n-1$ if $T$ is a star. The minimum is at least $2$ by Proposition \ref{p:th-1}, and is $2$ if $T = T'\circ K_1$ for some tree $T'$. Even if we restrict the tree to have maximum degree $d$, there are still trees with $\Omega(n)$ skew throttling numbers, e.g., when $T$ is obtained from a tree of maximum degree $d-2$ by adding two leaves to every vertex.  Although the skew throttling number  of the star is close to the (standard) throttling number, the skew throttling numbers  of paths and cycles behave more like the PSD throttling numbers of those graphs.  We begin with cycles and paths.

\begin{prop}\label{cycle}
For all $n \geq 3$, $\ths(C_n) = \lc \sqrt{2n}-\frac{1}{2} \rc$. 
\end{prop}
\bpf
The proof for the lower bound is the same as the proof of Proposition 2.5 in \cite{carlson2019throttling}. 

For the upper bound, we can start with the same construction as in the proof of Theorem 3.3 of \cite{carlson2019throttling}. We  initially color blue every $k+1$ vertex around the cycle, where $k$ is the largest even integer such that $n \geq \frac{k^2}{2}$. Let $r$ be the remainder when $n$ is divided by $k+1$. The initial blue coloring splits the white vertices into paths of size $k$ when $r=0$ and one short path containing $r-1$ white vertices when $r\ge 1$. Since $k$ is even, the white paths of size $k$ turn blue in $\frac{k}{2}$ rounds. If $r=0$ we are done and $\ths(C_n) \leq \lc \sqrt{2n}-\frac{1}{2} \rc$.  When $r\ge 1$ is odd, the number of white vertices in the short path is even, so the short path will also turn blue in at most $\frac{k}{2}$. We have used the same number of blue vertices as in the proof of \cite[Theorem 3.3]{carlson2019throttling}, so again $\ths(C_n) \leq \lc \sqrt{2n}-\frac{1}{2} \rc$. 

If $r\geq 1$ is even and $r < k$, then we can modify the initial coloring by increasing the lengths of $\frac{r}{2}$ of the white paths of length $k$ to $k+2$ and decreasing the length of the short path to $0$. Note that in this case, there are at least $\frac{k}{2}-1$ white paths of length $k$ before the modification because $n \geq \frac{k^2}{2}$. Moreover, we removed a blue vertex. Thus again we have $\ths(C_n) \leq \lc \sqrt{2n}-\frac{1}{2} \rc$. 

If $r = k$, then there are two possible values of $n$, namely $n = \lp\frac{k}{2}-1\rp(k+1)+k$ and  $n = \frac{k}{2}(k+1)+k$, since any other integer $m$ of the form $q(k+1)+k$ with integer $q$ has $m \geq \frac{(k+2)^2}{2}$ for $q > \frac{k}{2}$ and $m < \frac{k^2}{2}$ for $q < \frac{k}{2}-1$. For $n = \frac{k}{2}(k+1)+k$, we can again modify the initial coloring by increasing the lengths of the $\frac{r}{2}=\frac k 2$  white paths of length $k$ on the cycle to $k+2$, thereby decreasing the length of the short path to $0$. In this process we also removed a blue vertex, so we have $\ths(C_n) \leq \lc \sqrt{2n}-\frac{1}{2} \rc$.  For $n= \lp\frac{k}{2}-1\rp(k+1)+k$, we can modify the initial coloring by decreasing the lengths of  the white paths of length $k$ on the cycle to $k-2$, which decreases the propagation time by one.  The  number of white paths increases from $\frac k 2$ (consisting of $\frac k 2-1$ paths with $k$ white vertices and one short path with $k-1$ white vertices) to $\frac k 2+1$ (with each path having $k-2$ white vertices). Thus again we have $\ths(C_n) \leq \lc \sqrt{2n}-\frac{1}{2} \rc$.
\epf

\begin{prop}\label{path}
For all $n \geq 3$, $\ths(P_n) = \lc \sqrt{2(n+1)}-\frac{3}{2} \rc$. 
\end{prop}

\bpf
The proof for the lower bound is almost the same as the proof of Proposition 2.5 in \cite{carlson2019throttling}, except we use the inequality $s(2p+1)+2p \geq n$ instead of $s(2p+1) \geq n$ since the leaves can color their neighbors. 

For the upper bound, choose a set $S$ of blue vertices for the cycle on $n+1$ vertices that realizes $\ths(C_{n+1})=\lc\sqrt{2(n+1)}-\frac 1 2\rc$.  Delete a blue vertex $v$ that has no blue neighbor from the graph and the set $S$.  This results in the graph $P_n$ and set of blue vertices $S'=S\setminus \{v\}$ such that
\[\ths(P_n;S')=\lc\sqrt{2(n+1)}-\frac 1 2\rc-1=\lc\sqrt{2(n+1)}-\frac 3 2\rc.\qedhere\] 
\epf

As with standard and skew zero forcing, propagation time and throttling have been defined for PSD zero forcing (see Section \ref{s:intro} for the definitions of standard propagation time and throttling). The number of rounds to color the graph $G$ completely blue starting just with the vertices in $S$ blue is $\ptp(G;S)$ \cite{PSDpropTime}. For $G$ a graph and $S\subseteq V(G)$, define $\thp(G;S) = |S| + \ptp(G;S)$ and $\thp(G) = \min_{S \subseteq V(G)} \thp(G;S)$ \cite{carlson2019throttling}.  For paths and cycles skew throttling behaves more like PSD throttling than standard throttling, which for these graphs agrees with throttling for Cops and Robbers \cite{CR1}. 

\begin{rem} For a cycle $C_n$ with $n\ge 4$,
\bit
\item$\null$ \cite{carlson2019throttling} $\thz(C_n)= \begin{cases}
\lceil 2 \sqrt n - 1\rceil & \mbox{unless }n=(2k+1)^2 \\
2 \sqrt n & \mbox{if  }n=(2k+1)^2 
\end{cases}. $
\item $\null$ \cite{carlson2019throttling}
$\thp (C_n)= \left \lceil \sqrt{2n}-\frac{1}{2} \right \rceil.$
\item $\ths (C_n)= \left \lceil \sqrt{2n}-\frac{1}{2} \right \rceil.$
\eit

For a path $P_n$ with $n\ge 3$, 
\bit
\item$\null$ \cite{BY13throttling} $\thz (P_n)= \left \lceil 2\sqrt{n}-1 \right \rceil.  $
\item $\null$ \cite{carlson2019throttling}
$\thp (P_n)= \left \lceil \sqrt{2n}-\frac{1}{2} \right \rceil.$
\item $\ths (P_n)= \left \lceil \sqrt{2(n+1)}-\frac{3}{2} \right \rceil.$
\eit
\end{rem}

We use $T_{p,\ell}$ to denote the balanced spider with $p$ legs, each of length $\ell$; $T_{4,3}$ is shown in Figure \ref{f:T43spider}.  Observe that the order of $T_{p,\ell}$ is $n = p \ell + 1$. Note that $T_{p,1}=K_{1,p}$ and $\ths(K_{1,p})=p-1$ by Theorem \ref{t:high-th}, so the discussion here focuses on $\ell \ge 2$.

\begin{figure}[h]
\centering
\includegraphics[width=0.4\textwidth]{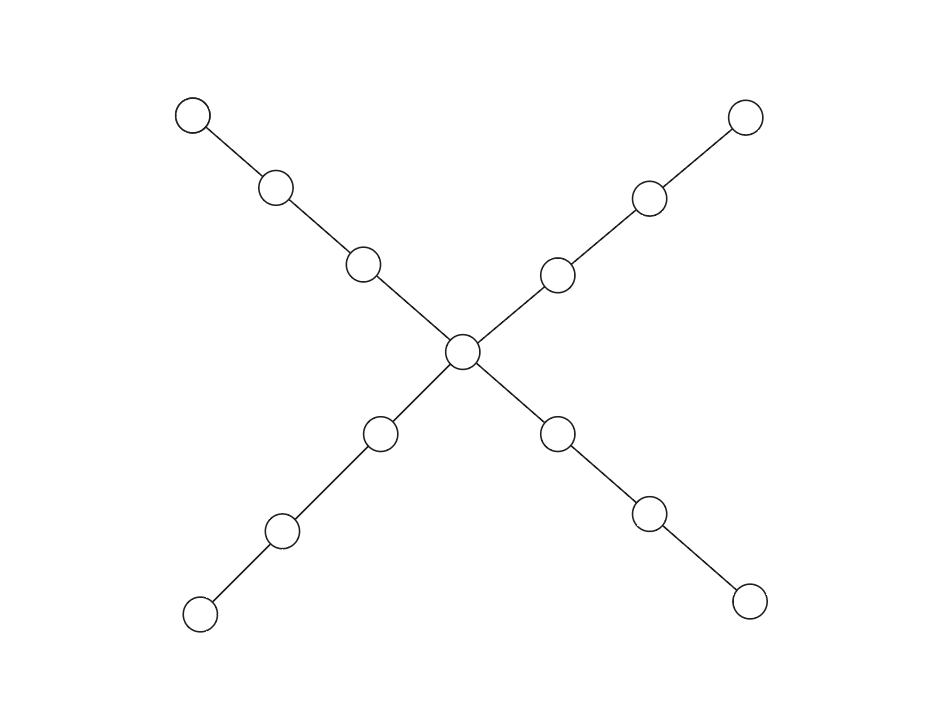}
\caption{$T_{4,3}$\label{f:T43spider}}
\end{figure}

\begin{thm}\label{t:balspider}
Let $\ell \ge 2$ be fixed and let $p  > \frac{\ell}{2} + 1$. Then
\[\ths(T_{p,\ell}) = 
\begin{cases}
1 + \frac{\ell}{2} &\text{ if }\ell \text{ is even}\\
1 +  p + \frac{\ell - 1}{4} &\text{ if } \ell = 4q+1 \text{ for some integer } q\\
1 +  p + \frac{\ell + 1}{4} &\text{ if } \ell = 4q+3 \text{ for some integer } q.
\end{cases}
\]
\end{thm}
\bpf
If $\ell$ is even, consider the skew zero forcing set consisting of one vertex, specifically the center. This partitions the graph into $p$ disjoint copies of $P_{\ell}$. Therefore, the graph will be colored in $\frac{\ell}{2}$ rounds. For the lower bound, note first that $\Zs(T_{p,\ell})=1$. If there is a leg with no initial blue vertex, then the propagation time is at least $\frac{\ell}{2}$ and $\ths(T_{p,\ell}) \geq \frac{\ell}{2} + 1$. If every leg has an initial blue vertex, then this would not be an optimal skew throttling set because $p > \frac{\ell}{2} + 1$. Thus  $\ths(T_{p,\ell}) = \frac{\ell}{2} + 1$ when  $\ell$ is even.

Now suppose $\ell$ is odd and let $S$ denote the set of vertices that are blue initially.  First note that if at least two legs of $T_{p,\ell}$ do not contain any vertex in $S$, then $S$ is not a skew zero forcing set. Thus $|S|\ge p-1$. We split the analysis into two cases, based on the parity of $\ell\mod 4$. 

First, suppose that $\ell = 4q+1$ for some positive integer $q$. Consider the skew zero forcing set $S$ consisting of the center $c$ and each vertex at distance $d$ from $c$ where $d=\frac{\ell+1}2$. Thus, $G-S$ has been partitioned into $2p$ disjoint copies of $P_{\frac{\ell - 1}2}$ and the graph will be colored in $\frac{\ell - 1}{4}$ rounds. For the lower bound,  if there is a leg that does not contain any vertex in $S$, then $\pts(T_{p,\ell};S)\ge \frac{\ell+1}{2}$, which  implies that  
\beq\label{eq:4qp1} \ths(T_{p,\ell};S)\ge p-1+\frac{\ell+1}{2} \geq 1+p+\frac{\ell - 1}{4}\eeq since $\ell \geq 5$. Thus there must exist an optimal initial blue set $S$ of size at least $p$ with a blue vertex in every leg, and we assume we have chosen such a $S$. If there is a leg that has only one vertex in $S$, then the  propagation time is at least $\frac{\ell - 1}{4}$, and this is achieved only when the center vertex is also blue. If every leg has at least two vertices in $S$, then $S$ would not be an optimal skew throttling set because \beq\label{eq:4qp1-2} p > \frac{\ell}{2} + 1>  1+\frac{\ell-1}{4}.\eeq Thus  $\ths(T_{p,\ell}) = 1+p+\frac{\ell - 1}{4}$ when $\ell = 4q+1$ for some positive integer $q$.

Now suppose that $\ell = 4q+3$ for some  integer $q\ge 0$.  It is straightforward to verify that $\ths(T_{p,3})= p+2=1+p+\frac{3 + 1}{4}$, so we assume $\ell\ge 7$. The argument is similar to the case $\ell = 4q+1$. For the upper bound, the blue vertices in the legs are placed at distance $\frac{\ell-1}2$ from the center and $G-S$ is partitioned into $p$ disjoint copies of $P_{\frac{\ell-3}{2}}$ and $p$ disjoint copies of $P_{\frac{\ell+1}{2}}$.  The lower bound argument is the same until \eqref{eq:4qp1}, where $\frac{\ell - 1}{4}$ is replaced by $\frac{\ell + 1}{4}$, but the equation remains true because now $\ell \ge 7$.  The analysis in \eqref{eq:4qp1-2} also remains valid with $\frac{\ell - 1}{4}$ replaced by $\frac{\ell + 1}{4}$.  So we assume  an optimal $S$ in which each leg has at least one vertex in $S$ and there is a leg with only one vertex in $S$.  
In $G-S$, a leg with exactly one vertex in $S$ is best  partitioned into one $P_{\frac{\ell-3}{2}}$ and one $P_{\frac{\ell+1}{2}}$ and the best possible propagation time is $\frac{\ell + 1}{4}$.  This can be achieved in two ways: When the center  is in $S$, or when the $P_{\frac{\ell-3}{2}}$ is next to the center and there is a blue vertex  at distance two from the center. In the latter case, there must be another vertex in $S$ in the leg with the blue vertex at distance two from the center.  Thus $|S|\ge  p+1$ and $\pts(T_{p,\ell};S)\ge \frac{\ell + 1}{4}$, or $|S|\ge  p$ and $\pts(T_{p,\ell};S)\ge \frac{\ell + 1}{4}+1$.  So $\ths(T_{p,\ell}) = 1+p+\frac{\ell + 1}{4}$ when $\ell = 4q+3$ for some positive integer $q$. \epf

\begin{thm}
For all  $\ell, p \ge 2$, $\frac{1}{2}f(p, \ell) \leq \ths(T_{p,\ell}) \leq 3 f(p, \ell)$, where
\[f(p,\ell) = 
\begin{cases}
\min(\ell, \sqrt{p \ell}) &\text{ if }\ell \text{ is even}\\
\max(p, \sqrt{p \ell}) &\text{ if } \ell \text{ is odd}.
\end{cases}
\]
\end{thm}

\begin{proof}
We split the proof into cases depending whether $p > \sqrt{p \ell}$ or $p \le \sqrt{p \ell}$.  In the case that $p > \sqrt{p \ell}$, then %$\sqrt{p \ell}>\ell$ and 
$p> \ell\ge 1+ \frac \ell 2$, so Theorem \ref{t:balspider} applies.  When $\ell $ is even, \[f(p,\ell)=\ell\ge \ths(T_{p,\ell})=1+ \frac \ell 2>\frac \ell 2= \frac 1 2 f(\ell,p).\]  When $\ell$ is odd, 
\[2f(p,\ell)=2p\ge 1 +  p + \frac{\ell + 1}{4}\ge \ths(T_{p,\ell})\ge 1 +  p + \frac{\ell - 1}{4}> \frac 1 2 p= \frac 1 2 f(\ell,p).\]

Now suppose $p \le \sqrt{p \ell}$, so $\sqrt{p \ell}\le\ell$ and $f(p,\ell)=\sqrt{p \ell}$.  For the lower bound, 
define $b$ to be the least number of initial blue vertices on any leg of the spider. If $b = 0$, then some leg has no initial blues, so the propagation time is at least $\frac{\ell}{2}\ge \frac 1 2 f(p,\ell)$. So assume $b > 0$.  Then there is a leg which has an uncolored interval between blue vertices of length at least $\frac{\ell-b}{b+1}$ by the pigeonhole principle, so the propagation time is at least $\frac{\ell-b}{4(b+1)}$. Since in this case there are at least $p b$ initial blue vertices and $p > 1$, we have $\ths(T_{p,\ell}) \geq p b + \frac{\ell}{4(b+1)} - \frac{1}{4} \geq \sqrt{\frac {p \ell}2}-\frac 1 4\ge \frac{\sqrt{p \ell}}{2}$.

For the upper bound,  use $2p \lc \frac{\ell}{2\lf\sqrt{p \ell} \rf+2} \rc \leq 2p \lc \frac{\ell}{2\sqrt{p \ell}} \rc \leq \frac{2\ell p}{\sqrt{p \ell}} \leq 2\sqrt{p \ell}$ initial blues arranged in adjacent pairs such that every white vertex is within distance $\sqrt{p \ell}$ of a blue vertex, for a propagation time of at most $\sqrt{p \ell}$.
\end{proof}

\begin{rem} Note that the method of the last proof can be used to obtain similar bounds for balanced spiders under other variants of throttling. In particular, the bound for the odd $\ell$ case in the last theorem has the same bound up to a constant factor as standard zero forcing throttling, while the even $\ell$ case has the same bound up to a constant factor as PSD zero forcing throttling and cop throttling. Specifically, we have $\thp(T_{p,\ell}) = \Theta(\min(\ell, \sqrt{p \ell}))$ for all positive $\ell$ and $p$, while $\thz(T_{p,\ell}) = \Theta(\max(p, \sqrt{p \ell}))$. It was shown in \cite{CR1} that the cop throttling number $\thc(T)$ is equal to $\thp(T)$ on all trees $T$, so we also get the same bounds up to a constant factor for $\thc(T_{p,\ell})$ as we do for $\thp(T_{p,\ell})$.
\end{rem} 

Our next result on hypercube graphs is an immediate corollary of the results of Kingsley \cite{King15}.

\begin{prop}{\rm \cite{King15}}
For $n \geq 2$, the $n^{\text{th}}$ hypercube $Q_n$  has $\Zs(Q_n) = 2^{n-1}$ and $\pts(Q_n) = 1$.
\end{prop}

\begin{cor}
For $n \geq 2$, the $n^{\text{th}}$ hypercube $Q_n$  has $\ths(Q_n) = 2^{n-1} + 1$. 
\end{cor}

\begin{prop}{\rm \cite{King15}}
For a connected graph $G \neq K_1$, the skew zero forcing number of the corona $G \circ K_1$ is $\Zs(G \circ K_1) = 0$ and $\pts(G \circ K_1) = 2$. 
\end{prop}

\begin{cor}
For a connected graph $G \neq K_1$, the skew throttling number of the corona $G \circ K_1$ is $\ths(G \circ K_1) = 2$. 
\end{cor}

\begin{prop}
For any graph $G$, $\ths(G \circ K_2) \leq |G| + 1$.
\end{prop}
\bpf
Consider the skew zero forcing set consisting of all vertices in $G$. Then, the remaining vertices which are $K_2$'s attached to each vertex in $G$ are forced in one round.
\epf

In general, $|G| + 1$ does not serve as a lower bound for $\ths(G \circ K_2)$. For example, suppose a connected graph $G$ contains $\ell \geq 3$ leaves. The set of all vertices of graph $G$ except for the leaves serves as a skew zero forcing set with the copies of $K_2$ attached to each initially blue vertex being forced in the first round, the leaves being forced in the second round, and finally the $\ell$ copies of $K_2$ are forced no later than the third and final round. Thus, $\ths(G \circ K_2) \leq |G| - \ell + 3 \leq |G| - 3 + 3 = |G|$.

Our next bound for graphs of fixed diameter is sharp, as shown by paths and cycles. However, we also find a much larger family of graphs that achieve this bound. %We define a \emph{skew zero forcing chain} $v_1, \dots, v_t$ to be a sequence of vertices such that $v_1$ is an initial blue vertex, and $v_{i+1}$ gets colored in the round after $v_i$ turns blue by some vertex that has $v_{i}$ as a neighbor. 
The {\em ball} $B(v,r)$ at vertex $v$ of radius $r$ in $G$ is all vertices at distance at most $r$ from $v$.
 
 \begin{lem}\label{lem:ball} Let $G$ be a graph, let $L=\{y_1,\dots,y_\ell\}$ denote the set of leaves of $G$, let $S=\{x_1,\dots,x_k\}\subseteq V(G)$, and let $t=\pts(G;S)$.  Then   
 \[V(G)= \bigcup_{i=1}^\ell B(y_i,2t)\ \bigcup \ \bigcup_{j=1}^k B(x_j,2t).\]
 \end{lem}
  \bpf A vertex  can perform a force in the first round if and only if it has at most one white neighbor.  Thus in order to force, a vertex must be a leaf (so it has only one neighbor) or it is a neighbor of a blue vertex. Thus any vertex colored blue in the first round must be at distance at most two from a blue vertex or a leaf, i.e. in some $B(y_i,2)$ or $B(x_j,2)$. This process is iterated through the $t$ rounds.   \epf

\begin{thm}\label{diamd}
For connected graphs $G$ of diameter $d\ge 4$ with minimum degree at least two, $\ths(G) \ge \sqrt{d}-\frac{1}{4}$. 
\end{thm}
\begin{proof} 
Suppose that $G$ has diameter $d$ and minimum degree at least two,  $S=\{x_1,\dots,x_k\}\subseteq V(G)$ is a skew zero forcing set for $G$ such that $\ths(G)=\ths(G;S)$, and  $t=\pts(G;S)$. By Lemma \ref{lem:ball}, $V(G)=  \bigcup_{j=1}^k B(x_j,2t)$  since $G$ has no leaves. Let $v_1$ and $v_{d+1}$ be vertices that have distance $d$ in $G$ and let $v_1, \dots, v_{d+1}$ be a shortest path between these vertices.  Consider the $k+1$ vertices $v_{1+i\lf\frac d k\rf}$ with $i=0,\dots,k$.
By the pigeonhole principle,  two of these vertices must be in the same $ B(x_j,2t)$, so $4t \geq \lf\frac{d}{k}\rf>\frac{d}{k}-1 $. Thus 
\[ \ths(G)=\ths(G;S)= k + t \geq k + \frac 1 4\lp\frac{d}{k}-1\rp \geq \sqrt{d}-\frac{1}{4}.\qedhere\]

\end{proof}

Besides paths and cycles, the families of graphs $C_n \circ K_2$ and $P_n \circ K_2$ also achieve the bound in Theorem \ref{diamd} up to a constant factor. % (i.e.,they are $\Theta(\sqrt{d})$). 
Using adjacent pairs of initial blue vertices on the cycle and path at intervals of approximately $\sqrt{n}$, the graphs $C_n \circ K_2$ and $P_n \circ K_2$ are colored in $\Theta(\sqrt{n})$ rounds using $\Theta(\sqrt{n})$ initial blue vertices. Figure \ref{f:CnoK2} illustrates this coloring   for $C_n \circ K_2$.

\begin{figure}[h]
\centering
\includegraphics[width=0.4\textwidth]{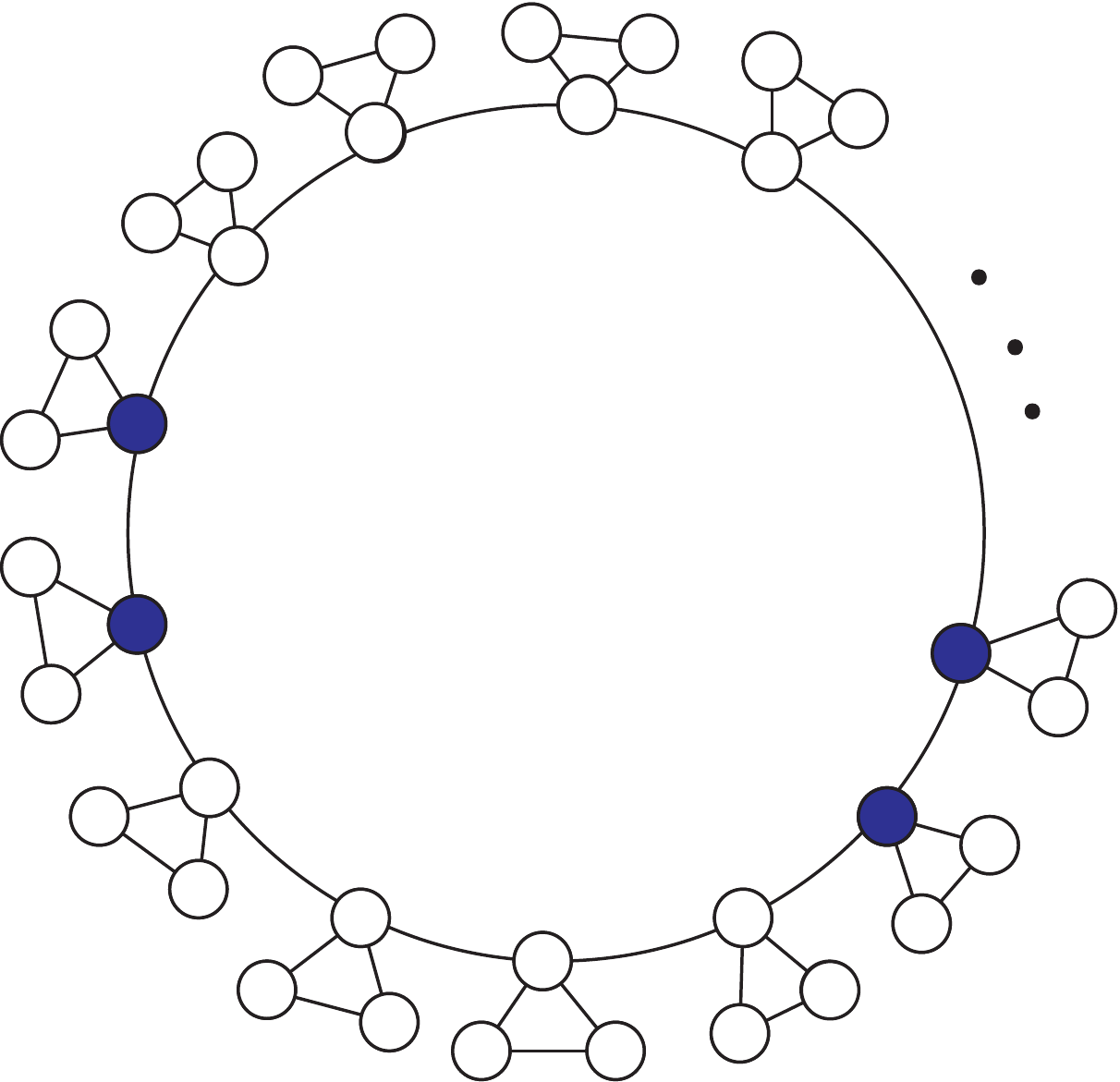}
%\vspace{-15pt}
\caption{An initial coloring that achieves $\ths(C_n \circ K_2) = O(\sqrt{n})$\label{f:CnoK2}}
\end{figure}

The same strategy can be used for a family of arbitrarily large order and  diameter $d$ to obtain the same $\Theta(\sqrt d)$ value, with the lower bound following from Theorem \ref{diamd}. Let $\mathcal{G}$ be the family of graphs that can be constructed by starting with a path or cycle and then for each vertex $v$ of the path or cycle, connecting any number  of copies (including none) of $K_2$ to $v$ (with one or both vertices of the $K_2$ adjacent to $v$). %The next theorem shows that the bound in Theorem \ref{diamd} is tight even when the order of the graph is much greater than the diameter. 

\begin{prop}
For any graph $G \in \mathcal{G}$ of diameter $d$, $\ths(G) = \Theta(\sqrt{d})$. 
\end{prop}

\begin{proof}
Observe that  the minimum degree of a graph $G\in \mathcal{G}$ is at least two, so the lower bound follows from Theorem \ref{diamd}. Suppose that $G\in \mathcal{G}$ is built by starting with a $P_n$ or $C_n$ subgraph. For the upper bound, place adjacent pairs of initial blue vertices on the cycle or path in $G$ at intervals of approximately $\sqrt{n}$, and also one at each at of the ends if $G$ is built from a $P_n$ subgraph. In the first round, all of the copies of $K_2$ attached to the initial blue vertices turn blue. By the second round, the neighbors of the initial blue vertices on the cycle or path turn blue. In the $(2i+1)^{\text{st}}$ round, copies of $K_2$ will turn blue if they are attached to the vertices on the cycle or path that turned blue in the $(2i)^{\text{th}}$ round. By the $(2i+2)^{\text{nd}}$ round, vertices on the cycle will turn blue if they are adjacent to vertices that turned blue in the $(2i)^{\text{th}}$ round. Thus the graph is colored entirely blue in approximately $\sqrt{n}$ rounds. 
\end{proof}

%An open problem is to characterize the graphs $G$ for which $|G| + 1 \leq \ths(G \circ K_2)$.\\

%%%%%%%%%%%%%%%%%%%%%%%%%%%%%%%%%%%%%%%%%%%
%%%%%%%%%%%%%%%%%%%%%%%%%%%%%%%%%%%%%%%%%%%


\begin{thebibliography}{20}
\bibliography{Meyniel_throttling}


\bibitem{AIM08} AIM Minimum Rank -- Special Graphs Work Group (F.~Barioli, W.~Barrett, S.~Butler, S.M.~Cioab\u{a}, D.~Cvetkovi\'c, S.M.~Fallat, C.~Godsil, W.~Haemers, L.~Hogben,  R.~Mikkelson,  S.~Narayan,  O.~Pryporova,   I.~Sciriha,  W.~So,   D.~Stevanovi\'c,  H.~van der Holst, K.~Vander Meulen,  A.~Wangsness).  Zero forcing sets and the minimum rank of graphs.   {\em Linear Algebra Appl.}, 428 (2008),  1628--1648.

\bibitem{smallparam}	F.~Barioli, W.~Barrett, S.~Fallat, H.T.~Hall, L.~Hogben, B.~Shader, P.~van den Driessche, and H.~van der Holst. Zero forcing parameters and minimum rank problems. {\em Linear Algebra Appl.}, 433 (2010), 401--411.

\bibitem{CR1} J.~Breen, B.~Brimkov, J.~Carlson, L.~Hogben, K.E.~Perry, C.~Reinhart. Throttling for the game of Cops and Robbers on graphs. Discrete Math., 341 (2018) 2418-2430.

\bibitem{BG} D.~Burgarth, V.~Giovannetti. Full control by locally induced relaxation. Phys. Rev. Lett. PRL 99 (2007), 100501.

\bibitem{BY13throttling} S.~Butler, M.~Young. Throttling zero forcing propagation speed on graphs. {\em Australas. J. Combin.}, 57 (2013), 65--71.

%\bibitem{PSDthrottle} J.~Carlson, L.~Hogben, J.~Kritschgau, K.~Lorenzen, M.S.~Ross, V.~Valle Martinez.  Throttling positive semidefinite zero forcing propagation time on graphs. Under review, available at \url{http://orion.math.iastate.edu/lhogben/research/CHKLRSV2017-PSDthrottle.pdf}.

%\bibitem{GTbook} R. Diestel. {\em Graph Theory}, 5th Ed.  Springer, Berlin, 2017.

\bibitem{carlson2019throttling} J.~Carlson, L.~Hogben, J.~Kritschgau, K.~Lorenzen, M.S.~Ross, S.~Selken, and V.V.~Martinez. Throttling positive semidefinite zero forcing propagation time on graphs. {\em Discrete Applied Mathematics}, 254 (2019), 33--46.

\bibitem{mr0} C. Grood, J.A. Harmse, L. Hogben, T. Hunter, B. Jacob, A. Klimas, S. McCathern, Minimum rank of zero-diagonal matrices described by a graph. {\it{Electron. J. Linear Algebra}}, {{27}} (2014), 458--477.

\bibitem{proptime} L.~Hogben, M.~Huynh, N.~Kingsley, S.~Meyer, S.~Walker, M.~Young. Propagation time for zero forcing on a graph. {\em Discrete Appl. Math.}, 160 (2012), 1994--2005.

\bibitem{IMA} IMA-ISU research group on minimum rank (M. Allison, E. Bodine, L.M. DeAlba, J. Debnath, L. DeLoss, C. Garnett, J. Grout, L. Hogben, B. Im, H. Kim, R. Nair, O. Pryporova, K. Savage, B. Shader, A. Wangsness Wehe). Minimum rank of skewsymmetric matrices described by a graph. Linear Algebra and its Applications, 432: 2457-2472, 2010.

\bibitem{King15} N.F.~Kingsley.  Skew propagation time. Thesis (Ph.D.), Iowa State  University, 2015.

\bibitem{S}  S.~Severini. Nondiscriminatory propagation on trees. J. Physics A, 41 (2008), 482-002 (Fast Track Communication).

\bibitem{Y} B.~Yang. Fast-mixed searching and related problems on graphs. Theoret. Comput. Sci. 507 (2013), 100-113.

%\bibitem{sage} Sage Minimum Rank Library Development Team. Sage Minimum Rank Library. Available at \url{https://github.com/jephianlin/mr_JG}.

\bibitem{PSDpropTime} N.~Warnberg. Positive semidefinite propagation time. {\em Discrete Appl. Math.}, 198 (2016), 274--290.

\end{thebibliography}
\end{document}